\documentclass[12pt]{article}
\usepackage[centertags]{amsmath}
\usepackage{amsfonts}
\usepackage{amssymb}
\usepackage{amsthm}
\input{amssym.def}
\usepackage{lineno}

\newtheorem{Definition}{Definition}[section]
\newtheorem{Theorem}[Definition]{Theorem}
\newtheorem{Lemma}[Definition]{Lemma}

\newtheorem{Corollary}[Definition]{Corollary}

\newcommand{\lc}{\mathcal{L}}
\newcommand{\rc}{\mathcal{R}}
\newcommand{\hc}{\mathcal{H}}
\newcommand{\jc}{\mathcal{J}}

\title{\Large \bf Nil-extensions of simple and right $\pi$-inverse ordered semigroups}

\author{A. Jamadar \\
\footnotesize{Department of Mathematics, Rampurhat College}\\
\footnotesize{Rampurhat-731224, West Bengal, India}\\
\footnotesize{amlanjamadar@gmail.com}}

\begin{document}

\date{}
\maketitle

\begin{abstract}{\footnotesize}
An ordered semigroup $S$ is right $\pi$-inverse if it is
$\pi$-inverse but not conversely. So the question arises under
what condition the converse holds. In this paper we study
nil-extensions of simple and right $\pi$-inverse ordered
semigroups and prove that $S$ is right $\pi$-inverse if and only
if $S$ is $\pi$-inverse in a $t$-Archimedean ordered semigroup.
Moreover, we characterize complete semilattice of nil-extensions
of simple and right $\pi$-inverse ordered semigroups.

 \end{abstract}

\section{Introduction and Preliminaries}
An ordered semigroup  is a partiality ordered set $(S,\leq)$ and
at the same time a semigroup $(S,\cdot)$ such that for all $a, \;b
, \;x \in S ,$ $a \leq b \;\textrm{implies} \;xa\leq xb
\;\textrm{and} \;a x \leq b x $. It is denoted by $(S,\cdot,
\leq)$. For a subset $A$ of $S$, let $(A]= \{x\in S: x\leq a,
\;\textrm{for some} \;a\in A\}$.

A nonempty subset $A$ of $S$ is  called a left (right) ideal
$\cite{Ke2006}$ of $S$, if $SA \subseteq A \;( A S \subseteq A)$
and $(A]= A$. A nonempty subset $A$ is called a (two-sided) ideal
of $S$ if it is both a left and a right ideal of $S$. An left
(right) ideal $I$ of $S$ is proper if $I \neq S$. $S$ is
left(right) simple if it does not contain proper left (right)
ideals. An ordered semigroup $S$ is called simple if for every
ideal $I$ of $S$, we have $I= S$. $S$ is called t-simple if it is
both left and right simple.

The principal $\cite{Ke2006}$ left ideal, right ideal, ideal and
bi-ideal generated by $a \in S$ are denoted by $L(a), \;R(a),
\;I(a)$and $B(a)$ respectively and defined by
$$L(a)= (a \cup Sa], \;R(a)= (a\cup aS], \;I(a)= (a\cup Sa \cup aS
\cup SaS] \;  and \;B(a)=(a \cup aSa].$$

Kehayopulu \cite{Ke2006} defined Greens relations $\lc, \;\rc,
\;\jc \;\textrm{and} \;\hc$ on an ordered semigroup $S$ as
follows:

$$ a \lc b   \; if   \;L(a)= L(b), \; a \rc b   \; if   \;R(a)=
R(b),\;
 a \jc b   \; if   \;I(a)= I(b) \;\textrm{and} \;\hc= \;\lc \cap \;\rc.$$

These four relations  are equivalence relations on $S$.

An element $a$ of $S$ is said to be regular if $a\in (aSa]$. We
denote set of regular elements by $Reg_{\leq}(S)$. An element $b
\in S$ is ordered inverse $\cite{HJ}$ of $a$ if $a \leq aba$ and
$b \leq bab$.  The set of all ordered inverses of an element $a
\in S$ is denoted by $V_\leq(a)$. An element $e\in S$ is said to
be ordered idempotent if $e\leq e^2$. The set of all ordered
idempotents of $S$ is denoted by $E_\leq (S)$.

An ordered semigroup $S$ is called Archimedean \cite{Cao 2000} if
for every $a, b \in S$ there is $m \in \mathbb{N}$ such that $b^m
\in (SaS]$. $S$ is called $r (l \;or \;t)$-Archimedean \cite{Cao
2000} if for every $a, b \in S$, there exists $m \in \mathbb{N}$
such that $b^m \in (aS] \;(b^m \in (Sa] \;\textrm{or} \;b^m \in
(aSa])$.

An ordered semigroup $S$ is called $\pi$-regular (resp. intra
$\pi$-regular) $\cite{Cao 2000}$ if for every $ a\in S$ there is
$m \in \mathbb{N}$ such that $a^{m} \in (a^{m}Sa^{m}]$ (resp. $a^m
\in (S a^{2m} S]$). We denote set of all $\pi$-regular and intra
$\pi$-regular elements by $\pi Reg_{\leq}(S)$ and $\Pi
Intra_\leq(S)$ respectively. A $\pi$-regular ordered semigroup $S$
is called right $\pi$-inverse \cite{AJ1} if for every $a\in S$,
there is $m \in \mathbb{N}$ such that any two inverses of $a^m$
are $\rc$-related.

Nil-extensions of an ordered semigroup $S$ with zero $0$ are
precisely the ideal extensions of an ideal $I$ of $S$ by the
nilpotent ordered semigroup $S/I$ \cite{ke2003}. Following Jamadar
\cite{AJ2}, a right $\pi$-inverse ordered semigroup is a
$\pi$-regular semigroup with the property that each element $a\in
S$ there exists $m\in \mathbb{N}$ such that any two inverses of
$a^m$ are $\rc$-related. These ordered semigroups are natural
generalization of  $\pi$-inverse ordered semigroups. Hansda and
Jamadar \cite{HJ1} studied ordered semigroups which are
nil-extensions of both simple and $\pi$-inverse ordered
semigroups. So it is a logical step to study ordered semigroups
which are nil-extensions of both simple and right $\pi$-inverse
ordered semigroups. Aim of this work is to describe nil-extensions
of simple and right $\pi$-inverse ordered semigroups. Furthermore
complete semilattice decompositions of the nil-extensions of
simple and right $\pi$-inverse ordered semigroups have been given
here. This paper is inspired by \cite{HB} and \cite{HJ1}.

 A congruence $\rho$ on $S$ is called semilattice if for all $a, b
\in S \;a \;\rho \;a^{2} \;\textrm{and} \;ab \;\rho \;ba$. A
semilattice congruence $\rho$ on $S$ is called complete if $a \leq
b$ implies $(a,ab)\in \rho$. The ordered semigroup $S$  is called
complete semilattice of subsemigroup of type $\tau$ if there exists
a complete semilattice congruence $\rho $ such that $(x)_{\rho}$ is
a type $\tau$ subsemigroup of $S$. Equivalently: There exists a
semilattice $Y$ and a family of subsemigroups $\{S\}_{\alpha \in Y}$
of type $\tau$ of $S$ such that:
\begin{enumerate}
\item \vspace{-.4cm} $S_{\alpha}\cap S_{\beta}= \;\phi$ for any
$\alpha, \;\beta \in \;Y \;with \; \alpha \neq \beta,$ \item
\vspace{-.4cm} $S=\bigcup _{\alpha \;\in \;Y} \;S_{\alpha},$ \item
\vspace{-.4cm} $S_{\alpha}S_{\beta} \;\subseteq \;S_{\alpha
\;\beta}$ for any $\alpha, \;\beta \in \;Y,$ \item \vspace{-.4cm}
$S_{\beta}\cap (S_{\alpha}]\neq \phi$ implies $\beta \;\preceq
\;\alpha,$ where $\preceq$ is the order of the semilattice $Y$
defined by \\$\preceq:=\{(\alpha,\;\beta)\;\mid
\;\alpha=\alpha\;\beta\;(\beta\;\alpha)\}$ $\cite{ke 2008}$.
\end{enumerate}

Let $S$ be a $\pi$-regular ordered semigroup. Due to Sadhya and
Hansda \cite{SH1}, the following equivalence relations $\lc^*$,
$\rc^*$, $\jc^*$ and $\hc^*$ are given by:

\begin{enumerate}
\item\vspace{-.4cm}
 $a\lc^* b\Leftrightarrow a^m\lc b^n$
\item\vspace{-.4cm} $a\rc^* b\Leftrightarrow a^m\rc b^n$
\item\vspace{-.4cm} $a\jc^* b\Leftrightarrow a^m\jc b^n$
\item\vspace{-.4cm} $\hc^*= \lc^*\cap \rc^*$
\end{enumerate}

where $a,b\in S$ and $m, n$ are the smallest positive integers
such that $a^m, b^n\in Reg_\leq (S)$.

For  $a,b \in S$, $a|b$ if and only if there exists $x, y \in S^1$
such that $b \leq xay$.

\begin{Theorem}\cite{AJ2}\label{500}
A $\pi$-regular ordered semigroup $S$ is a right $\pi$-inverse if
and only if for any two idempotents $e,f\in E_\leq(S)$, $e\lc^* f$
implies $e\rc^* f$.
\end{Theorem}

\begin{Theorem}\cite{AJ2}\label{15}
The following conditions are equivalent on a $\pi$-regular ordered
semigroup $S$.
\begin{enumerate}
\item\vspace{-.4cm} $S$ is right $\pi$-inverse;
\item\vspace{-.4cm} for $a\in S$ there is $m \in \mathbb{N}$ such
that $a', a'' \in V_{\leq}(a^m)$ \textrm{ implies}  $a'\rc a''$.
\end{enumerate}
\end{Theorem}
\begin{Theorem}\cite{HJ1}\label{74}
The following conditions on an ordered semigroup $S$ are
equivalent:
\begin{itemize}

\item[(i)]

$S$ is a nil-extension of  a left simple and $\pi$-inverse ordered
semigroup;

\item[(ii)]

$S$ is $\pi$-inverse and $l$-Archimedean ordered semigroup;

\item[(iii)]

$S$ is $\pi$-inverse and $a\lc^* b$ for every $a, b\in S$;

\item[(iv)]

$S$ is $\pi$-inverse and $e\lc^* f$ for every $e, f\in E_\leq(S)$;

\item[(v)]

$S$ is $\pi$-regular and $e\hc^* f$ for every $e, f\in E_\leq(S)$;

\item[(vi)]

$S$ is $\pi$-regular and $a\hc^* b$ for every $a, b\in S$;

\item[(vii)]

$S$ is $\pi$-inverse and $t$-Archimedean ordered semigroup;

\item[(viii)]

$S$ is a nil-extension of $t$-simple and $\pi$-inverse ordered
semigroup.

\end{itemize}
\end{Theorem}

\begin{Corollary}\cite{HJ1}\label{76}
The following conditions on an ordered semigroup $S$ are
equivalent:
\begin{itemize}
\item[(i)]

$S$ is a nil-extension of  a simple and $\pi$-inverse ordered
semigroup;

\item[(ii)]

$S$ is $\pi$-inverse and $e\jc^* f$ for all $e,f \in E_\leq(S)$.
\end{itemize}
\end{Corollary}

\begin{Lemma}\cite{Cao 2000}
 Let $S$ be an ordered semigroup and $I$ an ideal of $S$. Then
the following conditions are equivalent:
\begin{itemize}
\item[(i)]

$S$ is a nil-extension of $I$;

\item[(ii)]

$(\forall a\in S)(\exists m \in \mathbb{N}) \;a^m \in I$.
\end{itemize}
\end{Lemma}

\begin{Corollary}\cite{Cao 2000}\label{774}
The following conditions are equivalent on a poe-semigroup $S$:
\begin{itemize}
\item[(i)]

$S$ is a nil-extension of a $t$-simple po-semigroup;

\item[(ii)]

$S$ is a $t$-Archimedean po-semigroup in which $Intra(S)\neq
\phi$.
\end{itemize}
\end{Corollary}

\section{Nil-extensions of simple and right $\pi$-inverse ordered semigroups}
In this section we describe all ordered semigroups which are the
nil-extensions of simple and right $\pi$-inverse, left simple and
right $\pi$-inverse ordered semigroups. We define the sets
$\textbf{R}{V}_\leq (S)$ and $ \Pi \mathbf{R}{V}_\leq (S)$ as
follows:

$$\textbf{R}{V}_\leq
(S) =\{a\in S \;\mid \;\textrm{for} \;\textrm{any} \;x, \;y\in
V_\leq(a) \;\textrm{implies} \;x\rc y\;\},$$

$$ \Pi \mathbf{R}{V}_\leq (S) =\{a\in S \;(\exists m \in \mathbb{N})\;\mid \;\textrm{for} \;\textrm{any} \;x, \;y\in V_\leq(a^m)
\;\textrm{implies}
\;x\rc y\;\}.$$

\begin{Lemma}\label{ne51}
Let $S$ be an ordered semigroup. Then the following conditions are
equivalent on $S$:
\begin{itemize}
\item[(i)]

for all $a \in S$ and for all $c \in \mathbf{R}V_\leq(S)$, $a \mid
c$ implies $a^2 \mid c$;

\item[(ii)]

for all $a, b \in S$ and for all $c \in \mathbf{R}V_\leq(S)$, $a
\mid c \;and \;b \mid c$ implies $ab \mid c$.
\end{itemize}
\end{Lemma}
\begin{proof}
This follows from \cite{HJ1}.

\end{proof}

\begin{Theorem}\label{ne511}
Let an ordered semigroup $S$ be a complete semilattice $Y$ of
subsemigroups $\{S_\alpha\}_{\alpha \in Y}$. Then
\begin{itemize}

\item[(i)]

$\textbf{R}{V}_{\leq}(S)= \cup\textbf{R}{V}_{\leq}(S_\alpha)$.

\item[(ii)]

$S$ is right inverse if and only if $S_\alpha$ is right inverse
for all $\alpha \in Y$.

\item[(iii)]

$S$ is right $\pi$-inverse if and only if $S_\alpha$ is right
$\pi$-inverse for all $\alpha \in Y$.

\end{itemize}
\end{Theorem}
\begin{proof}
This follows from \cite{HJ1}.
\end{proof}

\begin{Lemma}\label{ne53}
Let an ordered  semigroup $S$  be a nil-extension of a semigroup
$K$ of type $\tau$ and $\mathbf{V}_\leq(S)\neq \phi $. Then the
following statements hold in $S$.
\begin{itemize}
\item[(i)] For every  $a \in \mathbf{R}V_\leq(S), \;a \in K$.
\item[(ii)]For every $\lc$-class of $S$ that contains an $a \in
\mathbf{R}V_\leq(S)$ is a subset of $K$.
\end{itemize}
\end{Lemma}
\begin{proof}
This follows from \cite{HJ1}.

\end{proof}

Following \cite{AJ1} we know that if $S$ is a $\pi$-inverse
ordered semigroup then it is right $\pi$-inverse but not
conversely. In the next theorem we describe ordered semigroups
which are nil-extensions of  simple and right $\pi$-inverse
ordered semigroups and show that $S$ is $\pi$-inverse if and only
if it is right $\pi$-inverse in a $t$-Archimedean ordered
semigroup. Also we prove that in a right $\pi$-inverse ordered
semigroup, $S$ is $l$-Archimedean if and only if it is
$t$-Archimedean.

\begin{Theorem}\label{1005}
The following conditions on an ordered semigroup $S$ are
equivalent:
\begin{itemize}
\item[(i)]

$S$ is a nil-extension of  a left simple and right $\pi$-inverse
ordered semigroup;

\item[(ii)]

$S$ is right $\pi$-inverse and $l$-Archimedean ordered semigroup;

\item[(iii)]

$S$ is right $\pi$-inverse and $a\lc^* b$ for any $a, b\in S$;

\item[(iv)]

$S$ is right $\pi$-inverse and $e\lc^* f$ for any $e, f\in
E_\leq(S)$;

\item[(v)]

$S$ is $\pi$-regular and $e\hc^* f$ for any $e, f\in E_\leq(S)$;

\item[(vi)]

$S$ is $\pi$-regular and $a\hc^* b$ for any $a, b\in S$;

\item[(vii)]

$S$ is $\pi$-inverse and $t$-Archimedean ordered semigroup;

\item[(viii)]

$S$ is right $\pi$-inverse and $t$-Archimedean ordered semigroup;

\item[(ix)]

$S$ is a nil-extension of  a $t$-simple and right $\pi$-inverse
ordered semigroup.

\end{itemize}
\end{Theorem}
\begin{proof}
$(i)\Rightarrow (ii)$: Let $S$ be a nil-extension of a left simple
and right $\pi$-inverse ordered semigroup $K$. Choose $a \in S$.
Then there exists $k \in \mathbb{N}$ such that $a^{k} \in K$.
Since $K$ is right $\pi$-inverse, for $a^{k}$ there exists $r \in
\mathbb{N}$ such that for any $x, y\in V_\leq (a^{m})\subseteq K$
implies $x\rc y$, where $m=kr$, by Theorem \ref{15}. Hence $S$ is
right $\pi$-inverse  ordered semigroup. Also $S$ is
$l$-Archimedean by Theorem \ref{74}.

$(ii)\Rightarrow (iii)$ and $(iii)\Rightarrow(iv) $:  These are
obvious.

$(iv)\Rightarrow (v)$: Since $S$ is right $\pi$-inverse so $S$ is
$\pi$-regular and $e\lc^* f$ implies $e\hc^* f$ for any $e, f\in
E_\leq(S)$, by Theorem \ref{500}. Hence $e\hc^* f$.

$(v)\Rightarrow (vi)$, $(vi)\Rightarrow(vii) $,
$(vii)\Rightarrow(viii)$ and $(viii)\Rightarrow(ix)$: This follows
from Theorem \ref{74}.

$(ix)\Rightarrow(i)$: Since $S$ is $t$-Archimedean so $S$ is a
nil-extension of a $t$-simple ordered semigroup $K$ by Corollary
\ref{774}. So $K$ is a left simple ordered semigroup. Let $a\in
K$. Since $S$ is right $\pi$-inverse, for $a\in S$ there exists $m
\in \mathbb{N}$ such that for any $a',a''\in V_\leq(a^m)$ implies
$a'\rc a''$. Now since $K$ is an ideal, so $a'a^ma'\in K$. Hence
$a'\leq a'a^ma'$ implies $a'\in K$. Similarly $a''\in K$. So
$a'\rc a''$ in $K$. Hence $K$ is a right $\pi$-inverse ordered
semigroup.

\end{proof}

\begin{Corollary}
The following conditions on an ordered semigroup $S$ are
equivalent:
\begin{itemize}
\item[(i)]

$S$ is a nil-extension of  a simple and right $\pi$-inverse
ordered semigroup;

\item[(ii)]

$S$ is right $\pi$-inverse and $e\jc^* f$ for all $e,f \in
E_\leq(S)$;

\item[(iii)]

$S$ is right $\pi$-inverse and $a\jc^* b$ for all $a,b \in S$;

\item[(iv)]

$S$ is right $\pi$-inverse and for all $a,b \in S$, there exists
$m \in \mathbb{N}$ such that $a^{m} \in (SbS]$;

\item[(v)]

$S$ is right $\pi$-inverse and Archimedean ordered semigroup.

\end{itemize}
\end{Corollary}
\begin{proof}
$(i)\Rightarrow (ii)$: Let $S$ be a nil-extension of a simple and
right $\pi$-inverse ordered semigroup $K$. Choose $a \in S$. Then
there exists $k \in \mathbb{N}$ such that $a^{k} \in K$. Since $K$
is right $\pi$-inverse, for $a^{k}$ there exists $r \in
\mathbb{N}$ and for any $x, y\in K$ such that $x, y\in V_\leq
(a^{m})$ implies $x\rc y$, where $m=kr$. Hence $S$ is right
$\pi$-inverse  ordered semigroup. Also $e\jc^* f$ for all $e,f \in
E_\leq(S)$, by Corollary \ref{76}.

$(ii)\Rightarrow (iii)$, $(iii)\Rightarrow (iv)$ and
$(iv)\Rightarrow (v)$: These are obvious.

$(v)\Rightarrow (i)$: Since $S$ is an Archimedean ordered
semigroup, so $S$ is nil-extension of the simple ordered semigroup
$K$ by Theorem $3.3$ of \cite{Cao 2000}. Let $a\in K$. Since $S$
is right $\pi$-inverse, for $a\in S$ there exists $m \in
\mathbb{N}$ such that for any $a',a''\in V_\leq(a^m)$ implies
$a'\rc a''$. Now since $K$ is an ideal, so $a'a^ma'\in K$. Hence
$a'\leq a'a^ma'$ implies $a'\in K$. Similarly $a''\in K$. So
$a'\rc a''$ in $K$. Hence $K$ is a right $\pi$-inverse ordered
semigroup

\end{proof}

\begin{Corollary}
An ordered semigroup $S$ is a nil-extension of  a right inverse
ordered semigroup if and only if the following conditions hold in
$S$:
\begin{itemize}
\item[(i)] $S$ is right $\pi$-inverse; \item[(ii)] for  $a \in S$
and $b \in \mathbf{R}V_\leq(S)$  such that   $a \leq ba$ implies
that $a \in \mathbf{R}V_\leq(S)$; \item[(iii)] for  $a \in S$ and
$b \in \mathbf{R} V_\leq(S)$ such that   $a\leq ab$ implies that
$a \in \mathbf{R}V_\leq(S)$; \item[(iv)] for $a \in S$ and $b \in
\mathbf{R} V_\leq(S)$ such that $a\leq b$ implies that $a
\in\mathbf{R} V_\leq(S)$.
\end{itemize}
\end{Corollary}

\begin{proof}
First suppose that $S$ is a nil-extension of a right inverse
ordered semigroup $K$.

$(i)$. Let  $ a \in S$.  Then  there is $m \in \mathbb{N}$ such
that $a^{m}\in K$. Since $K$ is right inverse, so there are
$x,y\in K$ such that for any $x, y\in V_\leq(a^m)$ implies $x\rc
y$. Thus $S$ is right $\pi$-inverse.

$(ii)$. Let $b \in \mathbf{R}V_\leq(S)$ and $a \in S$ such that $a
\leq ba$. Then  since $b \in Reg_\leq (S),\\ \;\textrm{there is}
\;z\in S$ such that $b\leq b(zb)^{n} \;\textrm{for all} \;n \in
\mathbb{N}$. Let $n_{1}\in \mathbb{N}$ such that $(zb)^{n_{1}} \in
K$. Then $b(zb)^{n_{1}} \in K$, which implies  $b\in K$ and so
$ba\in K$ and finally $a\in K$. Since $K$ is a right inverse
ordered semigroup, $a \in \mathbf{R}V_\leq (S)$.

$(iii)$. This is similar to $(ii)$.

$(iv)$. Let $a \in S$ and $b \in \mathbf{R}V_\leq (S)$ such that
$a \leq b$. Clearly $b \in Reg_\leq (S)$, and so for some  $z \in
S, \;b \leq b(zb)^n$, which holds for all $n \in \mathbb{N}$.
Since $S$ is nil-extension of $K$ there is  $m \in \mathbb{N}$
such that $(zb)^m \in K$ and so $b\in K$. Thus $a \in K$.
Therefore $a \in \mathbf{R} V_\leq (S)$.

Conversely, assume that given conditions  hold in $S$. Let $a\in
S$ be arbitrary. Then by (i) there exists $m\in \mathbb{N}$ such
that for any $x, y\in S$, $x, y\in V_\leq(a^m)$ implies $x\rc y$.
So $\mathbf{R}V_\leq(S)\neq \phi$. Say $T=\mathbf{R}V_\leq(S)$.
Thus for each $a \in S, \;\textrm{there exists} \;m\in \mathbb{N}$
such that for any $x, y\in S$, $x,y\in V_\leq(a^m)$ implies $x\rc
y$. Now choose $s\in S$ and $a\in T$. Then $a \in Reg_\leq (S)$,
there is $h \in S$ such that $a \leq a (ha)^n \;\textrm{for all}
\;n \in \mathbb{N}$. Let $m_{1}\in \mathbb{N}$ such that
$(ha)^{m_1} \in T$. So $sa \leq sa (ha)^{m_1} $ implies that $sa
\in V_\leq(S)=T$, by (iii). Similarly $as \in T$ follows from
(ii).

Now consider $a \in S$ and $b\in T$ such that $a\leq b$. Then  by
(iv) $a \in T$.  Also $T$ is a right inverse ordered semigroup.
Hence $S$ is nil-extension of a right inverse ordered semigroup
$T$.
\end{proof}

\section{Complete semilattice of nil-extensions of right $\pi$-inverse ordered  semigroups}
In this section we characterize  complete semilattice
decompositions of  nil-extensions of  right $\pi$-inverse ordered
semigroups.

\begin{Corollary}\label{1114}
Let $S$ be an ordered semigroup. Then the following conditions are
equivalent on $S$:
\begin{itemize}
\item[(i)]

$S$ is a  complete semilattice  of nil-extensions of simple and
right $\pi$-inverse ordered semigroups;

\item[(ii)]

$S$ is a complete semilattice of nil-extensions of simple
semigroups and $\Pi Intra_\leq(S)= \Pi \mathbf{R}{V}_\leq (S)$;

\item[(iii)]

$S$ is right $\pi$-inverse and is a complete semilattice of
Archimedean ordered semigroups.

\end{itemize}
\end{Corollary}

\begin{proof}
$(i)\Rightarrow (ii)$: Clearly  $S$ is a complete semilattice of
semigroups $\{S_\alpha \}_{\alpha \in Y}$, where $S_\alpha$ is a
nil-extension of  simple ordered semigroup $K_\alpha$ and $\rho$
be the corresponding  complete semilattice congruence on $S$. So
we need only to show that $\Pi Intra_\leq (S)= \Pi
\mathbf{R}{V}_\leq (S)$. For this, let $a \in \Pi
\mathbf{R}{V}_\leq (S)$. Then clearly there are  $m \in \mathbb{N}
\;\textrm{and} \;\alpha \in Y$ such that $a^m \in K_\alpha $ and
$a^m \in (K_\alpha a^{2m} K_\alpha]$. Thus $a \in \Pi
Intra_\leq(S)$.

Also let $b \in \Pi Intra_\leq (S)$. Then there are  $x, y \in S$
and $\gamma \in Y$ such that $b^m \leq xb^{2m} y$ and  $b  \in
S_\gamma$. Let  $S_\gamma$ is a nil-extension of simple and right
$\pi$-inverse ordered  semigroup $K_\gamma$.  Now $b^m \leq
(xb^m)^n b^m y^n$ for all $n \in \mathbb{N}$. Also by completeness
of $\rho$ we have that $(b^m)_\rho = (b^mxb^{2m}y)_\rho= (xb^m
xb^m y)_\rho=(xb^m)_\rho (xb^my)_\rho= (xb^m)_\rho (b^m)_\rho=
(xb^m)_\rho. $ Thus  $xb^m \in S_\gamma$. So  there is $m_1
\in\mathbb{N}$ such that $(xb^m)^{m_1} \in K_\gamma$. Thus
$b^{m_1} \in K_\gamma$. Since $K_\gamma$ is right $\pi$-inverse it
follows that $b \in \Pi \mathbf{R}{V}_\leq (S)$. Hence $\Pi
Intra_\leq (S)= \Pi \mathbf{R}{V}_\leq (S)$.

$(ii)\Rightarrow (iii)$:  Suppose $S$ is a complete semilattice of
semigroups $S_\alpha(\alpha \in Y)$, where $S_\alpha$ is a
nil-extension of  simple ordered semigroup $K_\alpha$ and $\Pi
Intra_\leq (S)= \Pi \mathbf{R}{V}_\leq (S)$. Let $a\in S$. Then
there are $m \in \mathbb{N}$ and $\alpha \in Y$ such that
$a^{m}\in K_\alpha$. Since each $K_\alpha$ is simple, for
$a^{m}\in K_\alpha$ there exists $r \in \mathbb{N}$ such that
$(a^m)^r \in (K_\alpha (a^m)^{2r} K_\alpha]$. Thus $ a\in \Pi
Intra_\leq (S)= \Pi \mathbf{R}{V}_\leq (S)$. Hence $S$ is right
$\pi$-inverse ordered semigroup. Also  $S$ is a complete
semilattice of Archimedean ordered semigroups, by Corollary 3.9 of
\cite{Cao 2000}.

$(iii) \Rightarrow (i)$: Suppose that the condition (iii) holds.
Then $S$ is a complete semilattice of nil-extensions of simple
semigroups, by Theorem $3.3$ of \cite{Cao 2000}.  Suppose that $S$
is a complete semilattice of semigroups $S_\alpha (\alpha \in Y)$,
where $S_\alpha$ is nil-extension of simple semigroup $K_\alpha$.
Let $a \in K_\alpha$. Since $S$ is $\pi$-inverse, there is $m \in
\mathbb{N}$ such that for any $z, y\in V_\leq(a^m)$ implies that
$z\rc y$, where $z, y\in S$. Now $a^m \leq a^m z a^m$ implies by
completeness of $\rho$, $(a^m)_\rho= (za^m)_\rho$ and so $a^m ,
za^m \in S_\alpha$. Now $a^m \leq a^m za^m$ implies that $a^m \leq
a^m (za^m z) a^m$. Since $S_\alpha$ is a nil-extension of
$K_\alpha$, $K_\alpha$ is an ideal of $S_\alpha$. Thus $za^m z \in
K_\alpha$. Now $z\leq za^mz$ implies by completeness of $\rho$,
$(z)_\rho= (za^mz)_\rho$. So $z\in K_\alpha$. Similarly $y\in
K_\alpha$. Hence $K_\alpha$ is right $\pi$-inverse and so $S$ is
complete semilattice of nil-extensions simple and right
$\pi$-inverse ordered semigroups.

\end{proof}

\begin{Corollary}
Let $S$ be an ordered semigroup. Then the following conditions are
equivalent on $S$:
\begin{itemize}
\item[(i)]

$S$ is a  complete semilattice  of nil-extensions of left simple
and right $\pi$-inverse ordered semigroups;

\item[(ii)]

$S$ is a complete semilattice of nil-extensions of left simple
semigroups and $\Pi Intra_\leq(S)= \Pi \mathbf{R}{V}_\leq (S)$;

\item[(iii)]

$S$ is right $\pi$-inverse and is a complete semilattice of
$l$-Archimedean ordered semigroups.

\end{itemize}
\end{Corollary}

\begin{proof}
This follows from Theorem \ref{1005} and Corollary \ref{1114}.

\end{proof}

\bibliographystyle{plain}

\end{document}